\documentclass[preprint]{elsarticle}
\usepackage[T1]{fontenc}
\usepackage[utf8]{inputenc}
\usepackage[letterpaper, margin=1in]{geometry}
\usepackage[colorlinks]{hyperref}
\usepackage{subfigure}
\usepackage{array}
\usepackage{amsmath,amssymb,amsfonts,xcolor,amsthm,dsfont,,mathabx}
\usepackage{mathtools, physics}
\usepackage{placeins}
\usepackage{csquotes}
\usepackage[capitalise, nameinlink, noabbrev, nosort]{cleveref}
\usepackage{doi}

\usepackage{latexsym}
\usepackage{euscript}
\usepackage{tocvsec2}
\usepackage{etoolbox,caption}
\usepackage{float}
\usepackage{marginnote}
\usepackage{epsfig,tikz}
\usepackage{tikz-cd}
\usetikzlibrary{patterns}
\usetikzlibrary{arrows.meta}
\usetikzlibrary{matrix}

\usepackage{enumitem}
\makeatletter
\newcommand{\mylabel}[2]{#2\def\@currentlabel{#2}\label{#1}}
\makeatother

\newtheorem*{theorem-no}{Theorem}

\numberwithin{pulse}{section}  

\makeatletter\g@addto@macro\bfseries{\boldmath}\makeatother

\makeatletter
\let\origsection\section
\renewcommand\section{\@ifstar{\starsection}{\nostarsection}}
\newcommand\sectionspace{\vspace{0.5ex}}
\newcommand\nostarsection[1]{\sectionspace\origsection{#1}\sectionspace}
\newcommand\starsection[1]{\sectionspace\origsection*{#1}\sectionspace}
\makeatother

\setlist[enumerate]{font=\normalfont}
\crefname{enumi}{}{}

\crefname{case}{case}{cases}
\creflabelformat{case}{#2#1#3}

\usepackage[hang, flushmargin]{footmisc}
\usepackage[english]{isodate}
\usepackage{mathrsfs}

\usepackage[capitalise, nameinlink, noabbrev, nosort]{cleveref}

\allowdisplaybreaks

\setenumerate{listparindent=\parindent}

\crefname{page}{page}{pages}

\numberwithin{equation}{section}
\crefname{equation}{equation}{equations}

\crefname{inequality}{inequality}{inequalities}
\creflabelformat{inequality}{#2(#1)#3}

\crefname{condition}{condition}{conditions}
\creflabelformat{condition}{#2(#1)#3}

\newtheorem{theorem}{Theorem}[section]

\newtheorem{thm}[theorem]{Theorem}
\crefname{thm}{Theorem}{Theorems}

\newtheorem{lem}[theorem]{Lemma}
\crefname{lem}{Lemma}{Lemmas}

\newtheorem{prop}[theorem]{Proposition}
\crefname{prop}{Proposition}{Propositions}

\newtheorem{cor}[theorem]{Corollary}
\crefname{cor}{Corollary}{Corollaries}

\newtheorem{assumption}[theorem]{Assumption}
\crefname{assumption}{Assumption}{Assumptions}

\theoremstyle{definition}

\newtheorem{defn}[theorem]{Definition}
\crefname{defn}{Definition}{Definitions}

\newtheorem{rmk}[theorem]{Remark}
\crefname{rmk}{Remark}{Remarks}

\newtheorem{example}[theorem]{Example}
\crefname{example}{Example}{Examples}

\newtheorem{claim}[theorem]{Claim}
\newtheorem*{claim*}{Claim}

\newcommand{\set}[1]{\left\{ #1 \right\}} 
\newcommand{\innprod}[1]{\left \langle #1 \right \rangle} 
\newcommand{\of}[1]{\left( #1 \right)} 

\newcommand{\Hil}{\mathcal{H}}
\newcommand{\fH}{\mathfrak{H}}

\newcommand{\fK}{\mathfrak{K}}
\newcommand{\fG}{\mathfrak{G}}
\newcommand{\F}{\mathbb{F}}

\newcommand{\G}{\mathcal{G}}
\newcommand{\N}{\mathbb{N}}
\newcommand{\C}{\mathbb{C}}

\newcommand{\U}{\mathcal{U}}
\newcommand{\B}{\mathcal{B}}
\newcommand{\unsp}{\G^{(0)}}
\newcommand{\prodsp}{\G^{(2)}} 
\newcommand{\ray}{\text{ray}}
\newcommand{\bxi}{\boldsymbol{\xi}}
\newcommand{\boldeta}{\boldsymbol{\eta}}
\newcommand{\RA}{R^u_{A(\G)}}

\newcommand{\calpi}{\varpi}

\renewcommand\paragraph[1]{\par\vspace{1em}\noindent\textbf{#1}}

\makeatletter
\def\ps@pprintTitle{%
  \let\@oddhead\@empty
  \let\@evenhead\@empty
  \let\@oddfoot\@empty
  \let\@evenfoot\@oddfoot
}
\makeatother

\begin{document}
\begin{frontmatter}
\title{On the restriction maps of the Fourier and Fourier-Stieltjes algebras over locally compact groupoids}

\author[mysecondaryaddress]{Joseph DeGaetani}
\author[mymainaddress]{Mahya Ghandehari\corref{mycorrespondingauthor}}
\cortext[mycorrespondingauthor]{Corresponding author}
\ead{mahya@udel.edu}
\address[mysecondaryaddress]{Department of Mathematics, State University of New York at Oswego, 
Oswego, NY, 13126, USA}
\address[mymainaddress]{Department of Mathematical Sciences, University of Delaware, Newark, DE, 19716, USA}
%
%
\begin{abstract}
The Fourier and Fourier-Stieltjes algebras over locally compact groupoids have been defined in a way that parallels their construction for groups.
In this article, we extend the results on surjectivity or lack
of surjectivity of the restriction map on the Fourier and Fourier-Stieltjes algebras of groups to the groupoid setting. In particular, we consider the maps that restrict the domain of these functions in the Fourier or Fourier-Stieltjes algebra of a groupoid to an isotropy subgroup. These maps are continuous contractive algebra homomorphisms. When the groupoid is \'{e}tale, we show that the restriction map on the Fourier algebra is surjective. The restriction map on the Fourier-Stieltjes algebra is not surjective in general. We prove that for a transitive groupoid with a continuous section or a group bundle with discrete unit space, the restriction map on the Fourier-Stieltjes algebra is surjective. 
We further discuss the example of an HLS groupoid, and obtain a necessary condition for surjectivity of the restriction map in terms of property FD for groups, introduced by Lubotzky and Shalom. As a result, we present examples where the restriction map for the Fourier-Stieltjes algebra is not surjective. 
Finally, we use the surjectivity results to provide conditions for the lack of certain Banach algebraic properties, including the (weak) amenability and existence of a bounded approximate identity,  in the Fourier algebra of \'{e}tale groupoids.
\end{abstract}
\end{frontmatter}


\section{Introduction}
The Fourier and Fourier-Stieltjes algebras of a general locally compact group $G$, denoted by $A(G)$ and $B(G)$, were introduced by Eymard in his seminal paper \cite{Eymard}. 
These Banach algebras consist of coefficient functions associated with certain group representations of $G$, and serve as foundational structures for studying Fourier analysis on locally compact groups.
For a locally compact Abelian group $G$, the Fourier and Fourier-Stieltjes algebras $A(G)$ and $B(G)$ can be identified, via the Fourier transform, with the group algebra and the measure algebra of the dual group $\widehat{G}$. 
In the case of a general (not necessarily Abelian) group $G$, the Fourier and Fourier-Stieltjes algebras are closely related to prominent operator algebras associated with $G$: 
specifically, $B(G)$ can be identified with the continuous linear dual to the group C*-algebra, and $A(G)$ serves as the predual to the group von Neumann algebra. 
These algebras thus play a central role in harmonic analysis on both Abelian and non-Abelian groups.

The definitions of Fourier and Fourier-Stieltjes algebras have been extended to settings in which the underlying structure is a locally compact groupoid rather than a group. Groupoids, first introduced by Brandt \cite{Brandt}, are algebraic structures that resemble groups but require only a partial product, and do not have a unique identity element. Instead, groupoids possess a unit space $\unsp$, whose elements serve as local identities. Each groupoid $\G$ admits range and source maps, $r$ and $s$, which map elements of $\G$ into the unit space $\unsp$.
For each $u\in \unsp$, the isotropy subgroup of $\G$ at $u$ is the subgroup $\G_u^u:=r^{-1}(\{u\})\cap s^{-1}(\{u\})$ that consists of elements that start and end at $u$. 
We present the necessary background material on groupoids in \Cref{subsec:groupoid}.

A theory of the Fourier and Fourier-Stieltjes algebras of locally compact groupoids has been developed in recent years from multiple perspectives \cite{Oty,PatersonPaper,Ramsay,Renault}. In this paper, we adopt the approach from \cite{PatersonPaper}, which develops a continuous theory for the Fourier and Fourier-Stieltjes algebras (in contrast to the `measurable' perspective taken in \cite{Renault}).
In this approach, a representation of a groupoid is defined as a family of unitary maps between Hilbert spaces from a continuous Hilbert bundle. Coefficient functions are then obtained using sections of the Hilbert field. The Fourier-Stieltjes algebra $B(\G)$ is the set of all coefficient functions of such representations (\Cref{def:BG}), while the Fourier algebra $A(\G)$ is defined using coefficient functions arising from the left regular representation of $\G$ (\Cref{def:FourierAlgebra}).
The definition of $A(\G)$ hinges on the existence a Haar system on $\G$ (\Cref{def:HaarSystem}); this is a collection of `local measures' that allow integration over the groupoid in a way similar to integrating over a group with a Haar measure.

Studying the Fourier and Fourier-Stieltjes algebras of locally compact groups 
provides insight into the properties of a locally compact group $G$ via the Banach algebraic features of $A(G)$ and $B(G)$.
A prominent result of Leptin states that a group $G$ is amenable precisely when $A(G)$ has a bounded approximate identity (\cite{Leptin}). 
For results regarding other important Banach algebraic properties of $A(G)$ and $B(G)$, in particular various forms of (Banach algebra) amenability, see, for example, \cite{ChoiGhandehariPaper1,ChoiGhandehariPaper2,Forrest-88,ForrestSameiSpronkPaper,Lee-etal-WA, LosertPaper,Runde-Spronk,Spronk2002,Spronk2010}.
Roughly speaking, an {amenable} commutative Banach algebra is one that lacks differentiability in a very strong sense.   For several important classes of Banach algebras, amenability identifies those with `good behavior'.   For example,  Connes~\cite{connes78} and  Haagerup~\cite{haagerup83}  showed that for C*-algebras amenability and nuclearity coincide. 
In \cite{Johnson}, Johnson showed that the Banach algebra amenability of the group algebra $L^1(G)$ is equivalent to the amenability of the group $G$ itself.
The motivation for extending this study to the Fourier and Fourier-Stieltjes algebras of locally compact groupoids lies in the fact that many C*-algebras naturally arise as algebras generated by groupoids. Thus, groupoids play a foundational role for these C*-algebras, making it valuable to explore the structure of Banach algebras associated with groupoids. Examples of such studies can be found in \cite{Bedos-Conti-dynamics, Bedos-Conti-2016, Finn-Sell, Renault-groupoid-cocyles}.

Many important Banach algebra properties, such as existence of a bounded approximate identity or various versions of amenability,  are \emph{hereditary} properties; that is, they are preserved under continuous algebra homomorphisms with dense range.
When $G$ is a locally compact group, a frequently used strategy in the study of hereditary Banach algebraic features of $A(G)$ and $B(G)$ is to reduce the problem to a suitable closed subgroup $H$ of $G$. In the case of the Fourier algebra, this approach is valid since the restriction map $r_{A}: A(G)\to A(H)$ is surjective \cite{Herz}. On the other hand, the restriction map $r_B:B(G)\to B(H)$ does not even have dense range in general. However, it has been shown that $r_B$ is surjective under additional conditions on $G$ or $H$, e.g., when $G$ is Abelian \cite[34.48]{HewittII} or SIN \cite{Cowling}, or when $H$ is open or compact or the connected component of the identity of $G$ \cite{Liukkonen,McMullen}. Similar restriction results for the Rajchman algebra $B_0(G)=B(G)\cap C_0(G)$ were discussed in \cite{MahyaThesisPaper}.

\subsection{Main contribution}
In the present article, we extend the results on surjectivity or lack of surjectivity of the restriction map on the Fourier and Fourier-Stieltjes algebras of groups to the groupoid setting. Here, we consider restrictions to isotropy subgroups of a given groupoid, rather than restricting to any subgroupoid. 
We do so in order to extend some of the known Banach algebraic results about the group Fourier and Fourier-Stieltjes algebras to certain classes of locally compact groupoids.
We first observe that the restriction of functions in $B(\G)$ (resp.~$A(\G)$) to an isotropy group $\G_u^u$ is a continuous and norm-decreasing algebra homomorphism $B(\G)\to B(\G_u^u)$ (resp.~$A(\G)\to A(\G_u^u)$) for every choice of $u\in \unsp$ (\Cref{thm:RestrictionTheorem}). When $\G$ is an \'{e}tale groupoid (i.e., its range map is a local homeomorphism), we show that the restriction map $A(\G)\to A(\G_u^u)$ is surjective (\Cref{thm:denserange}). We remark that the result of \Cref{thm:denserange} holds more generally for any locally compact groupoid admitting a Haar system that satisfies a certain condition (Condition $(*)$ in \Cref{def:condition*}). We conclude this section by \Cref{cor:hereditary} regarding hereditary Banach algebra properties, such as types of amenability of Banach algebras, that can be studied through the restriction map. This result provides a partial answer to a question of Paterson \cite{PatersonPaper} regarding criteria for the existence of bounded approximate identities for $A(\G)$.

Next, we observe that the restriction map for the Fourier-Stieltjes algebra of groupoids is not automatically surjective. Similar to the group setting, we restrict our attention to cases where we get surjective restriction maps, or a criteria for surjectivity. Namely, we show that the restriction map $B(\G)\to B(\G_u^u)$ is surjective, when $\G$ is a transitive groupoid with a continuous section (\Cref{thm:transitivesurjection}) or it is a group bundle with discrete unit space (\Cref{prop:bundlesurjection}). 
On the other hand, we discuss an interesting class of group bundles, called HLS groupoids, in \Cref{sec:HLSGroupoids}, and we obtain a necessary condition for surjectivity of the restriction map (\Cref{thm:HLS-not-surjective}) in terms of property FD introduced by Lubotzky and Shalom in \cite{Lubotzky}. As a corollary, we provide a class of examples of group bundles where the restriction $B(\G)\to B(\G_u^u)$ is {\it not} surjective (\Cref{cor:surjectivecounterexample}).

\subsection{ Organization of the paper}
In \Cref{chap:NotationsAndBackground}, we present the necessary background for
locally compact groupoids, their Haar systems and unitary representations, the
Fourier and Fourier-Stieltjes algebras for groupoids, and an overview of amenability properties of Banach algebras. 
In \Cref{sec:generalrestrictionresults}, we study the restriction maps $A(\G)\to A(\G_u^u)$ and $B(\G)\to B(\G_u^u)$, and show that when $\G$ is \'{e}tale, the map $A(\G)\to A(\G_u^u)$ is surjective. We devote \Cref{sec:specificrestrictionresults} to the study of surjectivity of the restriction map $B(\G)\to B(\G_u^u)$ for special classes of locally compact groupoids, namely the transitive groupoids (\Cref{sec:transitiverestrictions}), the group bundles with discrete unit spaces (\Cref{sec:discreteunitspacerestriction}), and the HLS groupoids (\Cref{sec:HLSGroupoids}). 
Finally, we conclude the paper by applying our results to obtain a decomposition of the Fourier algebra for certain group bundles in \Cref{sec:boxsum}.


\section{Notations and background}
\label{chap:NotationsAndBackground}

\subsection{Groupoids, topological groupoids, and Haar systems}\label{subsec:groupoid}
A \emph{groupoid} is a set $\G$ together with a subset $\prodsp\subseteq \G\times \G$, a product map $\prodsp \to \G$, $(a,b)\mapsto ab$ and an involutive inverse map $\G\to \G$, $a\mapsto a^{-1}$ that satisfies the following:
\begin{enumerate}
    \item if $(a,b),(b,c)\in \prodsp$, then $(ab,c),(a,bc)\in \G^{(2)}$ and $(ab)c=a(bc);$
    \item $(b,b^{-1})\in \prodsp$ for all $b\in \G$, and if $(a,b)\in \prodsp$, then $a^{-1}(ab)=b$, and $(ab)b^{-1}=a.$
\end{enumerate}
The key points that differentiate a group from a groupoid are the lack of an identity element and the lack of a full product (indeed, only pairs in $\prodsp$ are composable). Despite these differences, groupoids still enjoy some favorable properties of groups, such as cancellation \cite[Lemma 2.1.3]{Sims}.

Groupoids come equipped with \emph{range} and \emph{source} maps, denoted $r$ and $s$ respectively which are defined as 
$$r(a):=aa^{-1} \; \text{ and }\; s(a):=a^{-1}a.$$
These maps can be used to define the \emph{unit space} as $\unsp:=r(\G)=\set{aa^{-1}:a\in \G}$. Clearly, $\unsp=s(\G)$ as well. Next, we summarize some important properties of the concepts introduced so far.

\begin{prop}\cite[Lemma 2.1.4 and Corollary 2.1.6]{Sims}\label{prop:Sims2.1.4} Let $\G$ be a groupoid. Then the following are true:
\begin{enumerate}
    \item $(a,b)\in \prodsp$ if and only if $s(a)=r(b)$;
    \item $r(ab)=r(a)$ and $s(ab)=s(b)$ for all $(a,b)\in \prodsp$;
    \item $(ab)^{-1}=b^{-1}a^{-1}$ for every $(a,b)\in \prodsp$;
    \item $r(u)=s(u)=u$ for all $u\in \unsp$;
    \item $\unsp=\set{a\in \G: (a,a)\in \prodsp \text{ and } a^2=a};$
    \item \label{cancellation} If $(x,a),(x,b)\in \prodsp$ and $xa=xb$, then $a=b$. Similarly, if $(a,x),(b,x)\in \prodsp$ and $ax=bx$, then $a=b$.  
\end{enumerate}  
\end{prop}

For $u\in \unsp$, we set $\G^{u}:=r^{-1}\of{\set{u}}$, $\G_u:=s^{-1}\of{\set{u}}$, and $\G_u^v:=\G_u\cap \G^v$. These fibers do not typically form algebraic objects in their own right, however, the set $\G_u^u$ always forms a group, called the \emph{isotropy subgroup at $u$}.

\begin{example}\label{exa:groupoidexample}
Any group is a groupoid where the unit space is the singleton containing the group identity, and $\prodsp$ is the entire product space $\G\times\G$.
Next, we list a few well-known classes of groupoids that are not groups.
\begin{itemize}
\item[(i)] {\it Equivalence relations.}
    Let $R$ be any equivalence relation on a set $X$. Equip $R$ with the partial product $(x,y)(y,z)=(x,z)$ and inversion $(x,y)^{-1}=(y,x)$. 
    Then $R$ becomes a groupoid, with the unit space $R^{(0)}=\set{(x,x):x\in X}$.
\item[(ii)]{\it Transformation groupoids.} 
    Let $X$ be a set, and let $\Gamma$ be a group acting on $X$ by bijections. 
    Let $\G:=\Gamma\times X$. The groupoid product is defined for pairs of the form 
    $\left((g,h\cdot x),(h,x)\right)\in \G\times\G$ and is defined to be $(gh,x)$. The inverse is given by $(g,x)^{-1}=(g^{-1},g\cdot x)$. Then $\G$ becomes a groupoid with
    unit space $\unsp=\{(e_\Gamma,x):\ x\in X\}$.
\item[(iii)] {\it Group bundles.} \label{bundles} The disjoint union of groups, known as a \emph{group bundle}, also forms a groupoid. Namely, 
    for a collection $\{G_i\}_{i\in I}$ of groups, define $\G:=\sqcup_{i\in I} (\{i\}\times G_i)$. 
    Elements of $\G$ are composable if they have identical indices in their first coordinates. The product and the inverse are defined as $(i,x)(i,y)=(i,xy)$ and $(i,x)^{-1}=(i,x^{-1})$ respectively, where $xy$ and $x^{-1}$ are the product  and inverse in the group $G_i$.
    Then $\G$ becomes a groupoid with the unit space $\unsp=\{(i,e_{G_i}):\ i\in I\}$.
\end{itemize}    
\end{example}

A \emph{topological groupoid} is a groupoid $\G$ endowed with a topology in which
the inverse map $\G\to \G$, $\gamma \mapsto \gamma^{-1}$ and the product map $\prodsp\to \G$, $(g,h)\to gh$ are continuous. Here, $\prodsp$ is equipped with the relative topology as a subset of $\G\times \G$. 
A \emph{locally compact} (respectively \emph{Hausdorff}) groupoid is a topological groupoid whose topology is locally compact (respectively Hausdorff).
Clearly, the range and source maps of a topological groupoid are continuous. 
If we take $\G$ to be Hausdorff, then $\unsp$ is also Hausdorff, and consequently, the fibers $\G^u$ and $\G_u$ are closed subsets of $\G$.

Unless otherwise specified, the groupoids in this paper will be assumed to be both locally compact and Hausdorff.
Every locally compact group admits a Haar measure; this is a left (or right) translation invariant regular Borel measure that is finite on compact sets. 
The analog notion in the groupoid setting is a \emph{Haar system}.
 
\begin{defn}\cite[Definition 2.2.2]{PatersonBook}\label{def:HaarSystem}
    A (continuous) \emph{left Haar system} for a locally compact groupoid $\G$ is a family $\set{\lambda^u}_{u\in \unsp}$, where each $\lambda^u$ is a positive regular Borel measure on the locally compact Hausdorff space $\G^u$, such that the following three conditions are satisfied:
     \begin{enumerate}
     \item The support of each $\lambda^u$ is the whole of $\G^u$;
     \item\label{haarsystemfunction} For any $g\in C_c(\G)$, the function $g^{0}:\unsp \to \C$, defined $g^{0}(u)=\int_{\G^u}gd\lambda^u$ belongs to $C_c(\unsp)$;
     \item\label{haarsysteminvariance} For any $x\in \G$ and $f\in C_c(\G)$,
     $$\int_{\G^{s(x)}}f(xz)d\lambda^{s(x)}(z)=\int_{\G^{r(x)}}f(z)d\lambda^{r(x)}(z).$$
     \end{enumerate}
\end{defn}
In contrast to the group case,  Haar system of a groupoid need not be unique.
Assuming a left Haar system is available, we may define a groupoid convolution product as follows:
\begin{equation*}
    (f*g)(x)=\int_{\G^{r(x)}}f(y)g(y^{-1}x)d\lambda^{r(x)}(y), \mbox{ for }\ f,g\in C_c(\G).
\end{equation*}

Local compactness of a groupoid is not enough to guarantee the existence of a Haar system (see \cite{SedaPaper}). It is well-known that if a locally compact groupoid $\G$ possesses a Haar system, then its range and source maps $\G\to\unsp$ will be open (see, e.g., \cite[Proposition 1.23]{ToolKit}). 
However, it is not known if this necessary condition is sufficient in general.  
For special cases, e.g., for second countable locally compact transitive groupoids \cite{SedaPaper,Williams-Haar} or locally compact group bundles \cite[Lemma 1.3]{ToPDown}, the groupoid has a Haar system exactly when the range map $r:\G\to \unsp$ is open.  

While not all locally compact groupoids have Haar systems, \'{e}tale groupoids (defined below) always have a straightforward and easily described one.
\begin{defn}
    A locally compact, Hausdorff groupoid will be called \emph{\'{e}tale} if the range map $r:\G\to \G$ (or equivalently, the source map) is a local homeomorphism.
\end{defn}
All discrete groupoids are \'{e}tale.
All \'{e}tale groupoids have Haar systems consisting of counting measures. \'{E}tale groupoids have received a great deal of attention (see, for example, \cite{MR4533452,Brown-etal-2014,MR4592883,MR4510931,MR4645718,MR4608439,Matui-PAMS-2012,MR4483992,MR4563262,Pedro}).
%

\subsection{Fourier and Fourier-Stieltjes algebras for groupoids}
\label{sec:groupoidalgebras}
To extend the definition of Fourier and Fourier-Stieltjes algebras associated with locally compact groups to the realm of groupoids, we start with defining what a representation of a locally compact groupoid is. 
Roughly speaking, while a group requires unitary representations over Hilbert spaces, groupoids will require $\G$-Hilbert bundles.

\begin{defn}[Continuous Banach field] \cite[Definition 10.1.2]{Dixmier}
\label{def:BanachField}
    A \emph{continuous field of Banach spaces} $\mathscr{E}$ over a topological space $T$ is a family $(E_t)_{t\in T}$ of Banach spaces, with a set $\Gamma \subseteq \Pi_{t\in T} E_t$ such that:
    \begin{enumerate}
        \item \label{linearspace} $\Gamma$ is a complex linear subspace of $\Pi_{t\in T} E_t$;
        \item \label{density} For every $t\in T$, the set $\set{x(t):x\in \Gamma}$ is dense in $E_t$;
        \item \label{normcontinuity} For every $x\in \Gamma$, the function $t\to \norm{x(t)}_{E_t}$ is continuous;
        \item \label{continuoussection} Let $x\in \Pi_{t\in T} E_t$.  If for every $t\in T$ and every $\varepsilon>0$, there exists $x'\in \Gamma$ such that $\norm{x(\tau)-x'(\tau)}_{E_\tau}<\varepsilon$ for all $\tau$ in some neighborhood of $t$, then $x\in \Gamma$.
    \end{enumerate}
    We refer to elements of $\Pi_{t\in T} E_t$ as \emph{sections}, and elements of $\Gamma$ as \emph{continuous sections}.
    In the case where each Banach space is also a Hilbert space, the definition above is referred to as a \emph{continuous Hilbert field}.
\end{defn}

In practice, continuous fields of Banach spaces are built as a `closure' of a collection of sections satisfying \eqref{linearspace}, \eqref{density} and \eqref{normcontinuity} of \Cref{def:BanachField}. This is done using the following result.

\begin{prop}\cite[Proposition 10.2.3]{Dixmier}\label{prop:fundamentalsections}
    Let $\set{E_t}_{t\in T}$ be a family of Banach spaces indexed over a topological space $T$. Suppose that there exists a set $\Xi \subseteq \prod_{t\in T}E_t$ of sections satisfying the first three conditions of \Cref{def:BanachField}. Then there is a unique subset $\Gamma\subseteq \prod_{t\in T}E_t$ that satisfies all four conditions and contains $\Xi$.

    The collection $\Xi$ is referred to as the \emph{fundamental family of sections}, and elements in $\Xi$ are called the \emph{fundamental sections}. 
\end{prop}

\Cref{prop:fundamentalsections} enables the definition of direct sums and tensor products for continuous Hilbert fields. Let $\fH := (\set{\fH_t}_{t\in T}, \Gamma)$ and $\fK := (\set{\fK_t}_{t\in T}, \Gamma')$ be two continuous Hilbert fields over a topological space $T$. The \emph{direct sum of Hilbert fields}, denoted by $\fH \oplus \fK$, is the continuous Hilbert field with fibers $\set{\fH_t \oplus \fK_t}_{t\in T}$ and a fundamental family of sections $\Gamma \oplus \Gamma'$, where each section is defined by $(f \oplus g)(t) = f(t) \oplus g(t)$. Similarly, the \emph{tensor product of Hilbert fields}, denoted by $\fH \otimes \fK$, is the continuous Hilbert field with fibers $\set{\fH_t \otimes \fK_t}_{t\in T}$ and a fundamental family of sections $\Gamma \otimes \Gamma'$, where each section is given by $(f \otimes g)(t) = f(t) \otimes g(t)$.

\begin{example}\label{exp-fields}
We now describe important examples of continuous fields of Banach spaces.
    \begin{enumerate}
        \item\label{exp-discrete} If $T$ is discrete, then $\Gamma=\Pi_{t\in T} E_t$, and the continuous field coincides with the direct product of the Banach spaces.
        \item\label{exa:contfield} Suppose $E_t=E$ for every $t\in T$. In this case, elements of $\prod_{t\in T} E$ can be identified with functions from $T$ to $E$. Let $C(T,E)$ denote the collection of continuous functions from $T$ to $E$. The \emph{constant field} is the continuous Banach field whose set of fundamental sections is $C(T,E)$ (\cite[pg 7]{PatersonPaper}).
       It is easy to verify that if $f\in \Pi_{t\in T} E_t$ satisfies Condition \eqref{continuoussection} of \Cref{def:BanachField}, then $f$ must belong to $C(T,E)$. That is, every continuous section is a fundamental section in this particular example.
       \item\label{exa:L2(G)} Let $\G$ be a locally compact groupoid with a Haar system 
       $\{\lambda^u\}_{u\in \unsp}$. Set $T=\unsp$ and $E_u=L^2(\G^u,\lambda^u)$ for each $u\in\unsp$. Every function $f\in C_c(\G)$ can be viewed as a section, in the sense that $u\mapsto f^u=f|_{\G^u}$. One can verify that $C_c(\G)$ forms a fundamental family of sections as defined in \Cref{prop:fundamentalsections}, and consequently, uniquely determines a set of continuous sections, which constitutes the left regular field denoted by $L^2(\G)$.
       \cite[pg. 10]{PatersonPaper} provides an alternative definition of $L^2(\G)$ as follows: a function $F:\G\to \C$ is a continuous section of the left regular field if and only if $u\mapsto \norm{F^u}$ is continuous and $u\mapsto\innprod{F^u,g^u}$ is continuous for all $g\in C_c(\G)$.
    \end{enumerate}
\end{example}

Moving forward, let $\fH:=(\set{\fH_t}_{t\in T},\Gamma)$ be a continuous Hilbert field over the space $T$. Denote the norm of each Hilbert space fiber $\mathfrak{H}_t$ as $\|\cdot\|_t$ or $\|\cdot\|_{\mathfrak{H}_t}$ for $t\in T$. 
We define the \emph{section norm}
$$\norm{\xi}_\Delta=\sup_{t\in T}\norm{\xi(t)}_t.$$
Let $\Delta_b(\fH)$ denote the subspace of $\Gamma$ consisting of sections whose $\|\cdot\|_\Delta$-norm is bounded.  Equipped with pointwise operations and $\|\cdot\|_\Delta$-norm, the space $\Delta_b(\fH)$ is a Banach space (\cite[pg 6]{PatersonPaper}).
We further define  $\Delta_c(\fH)$ to be the subspaces of $\Gamma$ of continuous sections with compact support, and $\Delta_0(\fH)$ to be those sections in $\Gamma$ that vanish at infinity under the section norm. All three of $\Delta_b,\Delta_c$, and $\Delta_0$ are $C_0(T)$-modules under pointwise multiplication as the module action (\cite[Proposition 10.1.9]{Dixmier}).

We now give the definition of a groupoid representation as presented in \cite[Section 3]{PatersonPaper}.

\begin{defn}[Groupoid representation]\label{def:ghilbdle}
    Let $\G$ be a topological groupoid, and $\fH:=(\set{\fH_u}_{u\in \unsp},\Gamma)$ be a continuous Hilbert field. A \emph{(continuous unitary) representation of $\G$ on the continuous Hilbert field $\fH$} is a collection of unitary maps $\calpi(x):\fH_{s(x)}\to \fH_{r(x)}$, for each $x\in \G$,  that satisfy the following:
    \begin{enumerate}
        \item \label{coefficientcontinuity} For each $\bxi, \boldeta \in \Delta_b(\mathfrak{H})$ the map $\G\to \C$ defined as $x\mapsto \innprod{\calpi(x)\bxi(s(x)),\boldeta(r(x))}_{\fH_{r(x)}}$ is continuous;
        \item  For every $(x,y)\in \prodsp$, we have $\calpi(xy)=\calpi(x)\calpi(y)$. 
    \end{enumerate}
    A representation as above is also called a \emph{$\G$-Hilbert bundle over $\unsp$}. We use these two terminologies interchangeably. 
\end{defn}

The map $x\mapsto \innprod{\calpi(x)\bxi(s(x)),\boldeta(r(x))}_{\fH_{r(x)}}$ given in \Cref{def:ghilbdle} is referred to as the \emph{coefficient of $\fH$ associated with $\bxi$ and $\boldeta$}, and denoted by $\fH_{\bxi,\boldeta}$.  To differentiate between coefficient functions of group representations, and coefficients of $\G$-Hilbert bundles, we will reserve lower case Greek letters and unbolded symbols for representations, and accented text with bold symbols to refer to $\G$-Hilbert bundles and their sections.
We note that it suffices to check condition \eqref{coefficientcontinuity} of \Cref{def:ghilbdle} for $\bxi,\boldeta\in \Xi$, the fundamental sections, that are $\|\cdot\|_\Delta$-bounded.
 
Let $T=\unsp$. Let $\fH:=(\set{\fH_t}_{t\in T},\Gamma)$ and $\fG:=(\set{\fG_t}_{t\in T},\Gamma')$ be two continuous fields of Hilbert spaces over $T$. 
Suppose $\calpi$ and $\varrho$ are representations of $\G$ on $\fH$ and $\fG$, respectively.
A collection of bounded linear operators $\set{T_t:\fH_t\to \fG_t}_{t\in T}$ induces a map $T:\fH\to \fG$ defined as $(T(\bxi))(t)=T_t(\bxi(t))$. We say $\set{T_t:\fH_t\to \fG_t}_{t\in T}$ is a \emph{morphism} of $\fH$ and $\fG$ if $\set{T(\bxi):\bxi \in \Gamma}\subseteq \Gamma'$ (see \cite[Definition 2.8]{Bos}). A morphism is called an \emph{intertwining morphism} if $\varrho(x)T_{s(x)}=T_{r(x)}\calpi(x)$ for every $x\in \G$. 
Two representations are \emph{unitarily equivalent}, denoted $\fH\simeq \fG$, if there is an intertwining morphism consisting of unitary operators between them (Figure~\ref{fig:morphism}). 
%
\begin{figure}[h]
    \centering
    \[\begin{tikzcd}
	{\fH_{s(x)}} && {\fH_{r(x)}} \\
	\\
	{\fG_{s(x)}} && {\fG_{r(x)}}
	\arrow["{\calpi(x)}", from=1-1, to=1-3]
	\arrow["{\varrho(x)}"', from=3-1, to=3-3]
	\arrow["T_{s(x)}"', from=1-1, to=3-1]
	\arrow["T_{r(x)}", from=1-3, to=3-3]
\end{tikzcd}\]
    \caption{An intertwining morphism between $\fH:=(\set{\fH_t}_{t\in T},\Gamma)$ and $\fG:=(\set{\fG_t}_{t\in T},\Gamma')$.}
    \label{fig:morphism}
\end{figure}

As in the case of group representations, the direct sums and tensor products of $\G$-Hilbert bundles are themselves $\G$-Hilbert bundles \cite[pg 7]{PatersonPaper}. 
Let $\fH$ and $\fK$ be two continuous Hilbert fields over the topological space $\unsp$, and $\calpi$ and $\varrho$ be representations of $\G$ on $\fH$ and $\fG$ respectively. Equip the continuous Hilbert field $\fH\oplus \fK$ with linear actions $(\calpi \oplus \varrho)(x):\fH_{s(x)}\oplus \fK_{s(x)}\to\fH_{r(x)}\oplus \fK_{r(x)}$ given by $(\calpi(x)\oplus \varrho(x))(u\oplus v)=\calpi(x)u\oplus \varrho(x)v$. With these linear actions, $\fH\oplus \fK$ becomes a $\G$-Hilbert bundle.
Similarly, $\fH\otimes \fK$ is the $\G$-Hilbert bundle with linear actions $(\calpi \otimes \varrho)(x):\fH_{s(x)}\otimes \fK_{s(x)}\to\fH_{r(x)}\otimes \fK_{r(x)}$ given by $(\calpi(x)\otimes \varrho(x))(u\otimes v)=\calpi(x)u\otimes \varrho(x)v$.

\begin{example}(The trivial representation)\label{exa:trivialbundle}
Fix a Hilbert space $\Hil$, and for every $u\in \unsp$, let $\mathfrak{H}_u=\Hil$.
Consider the constant field $\Gamma=C(\unsp,\Hil)$ as defined in \Cref{exp-fields}  \eqref{exa:contfield}.
For every $x\in \unsp$, let $\calpi(x)=I_{\Hil}$. Clearly, this forms a $\G$-Hilbert bundle. 
\end{example}

\begin{example}(The left regular representation)\label{exa:LeftRegularBundle}
    Let $\G$ be a locally compact groupoid with a Haar system $\set{\lambda^u}_{u\in \unsp}$. Consider the continuous Hilbert field $L^2(\G):=\set{L^2(\G^u,\lambda^u)}_{u\in \unsp}$ as presented in  \Cref{exp-fields} \eqref{exa:L2(G)}. 
    For $x\in \G$, define 
    $$L_x: L^2(\G^{s(x)},\lambda^{s(x)})\to L^2(\G^{r(x)}, \lambda^{r(x)}), \quad (L_xf)(z)=f(x^{-1}z).$$ 
    This will turn $L^2(\G)$ into a $\G$-Hilbert bundle, denoted $\Lambda$. 
\end{example}

We are now able to define the central objects of this manuscript. 
\begin{defn}\label{def:BG}
    The \emph{Fourier-Stieltjes algebra} of a locally compact groupoid $\G$ is defined as
    $$B(\G):=\set{\fH_{\bxi,\boldeta}:\fH \text{ is a $\G$-Hilbert Bundle, }\bxi,\boldeta\in \Delta_b(\fH)},$$
   equipped with pointwise operations and the norm defined as
   $$\norm{\varphi}_{B(\G)}:=\inf_{\varphi=\fH_{\bxi,\boldeta}}\norm{\bxi}_\Delta\cdot{\norm{\boldeta}_\Delta}$$
where the infimum is taken over all coefficient representations of $\varphi=\fH_{\bxi,\boldeta}$.
\end{defn}
Clearly, $B(\G)$ is an algebra, since $\fH_{\bxi,\boldeta}+\fH'_{\bxi',\boldeta'}=(\fH\oplus \fH')_{\bxi\oplus \bxi',\boldeta\oplus\boldeta'}$ and\linebreak $\fH_{\bxi,\boldeta}\fH'_{\bxi',\boldeta'}=(\fH\otimes \fH')_{\bxi\otimes \bxi',\boldeta\otimes\boldeta'}$. 
It has been shown that $B(\G)$ is in fact a unital commutative Banach algebra \cite{PatersonPaper}.

\begin{defn}\label{def:FourierAlgebra}
Let $\G$ be a locally compact groupoid with a left Haar system $\{\lambda_u\}_{u\in\unsp}$, and $\Lambda$ denote the left regular representation of $\G$
defined on the continuous Hilbert field $L^2(\G)=\set{L^2(\G^u,\lambda^u)}_{u\in \unsp}$.
Define
$$A_{cf}(\G)=\left\{\Lambda_{f,g}:\ f,g\in C_c(\G)\right\}.$$
Furthermore, let $A_c(\G)$ be the algebra generated by $A_{cf}(\G)$ under pointwise operations. The \emph{Fourier Algebra} of a locally compact groupoid is defined as $$A(\G):=\overline{A_c(\G)}^{\norm{\cdot}_{B(\G)}}.$$
\end{defn}
It is easy to see that both $A_{cf}(\G)$ and $A_c(\G)$ are subsets of $C_c(\G)$. 
By definition, $A(\G)$ is a closed subalgebra of $B(\G)$, thus $A(\G)$ is itself a Banach algebra. 
\begin{assumption}\label{assump-etale}
When defining the Fourier algebra for an \'{e}tale groupoid, we equip the groupoid with the collection of counting measures as the canonical choice for its Haar system.
\end{assumption}
\begin{rmk}[Relation with the Fourier and Fourier-Stieltjes algebras of groups]
When $\G$ is a group, then $B(\G)$ and $A(\G)$ coincide with the Fourier-Stieltjes and Fourier algebras in the group case.
These algebras for a locally compact group $G$ are defined through continuous unitary representations, i.e., group homomorphisms $\pi:G\to \U(\Hil)$, for some Hilbert space $\Hil$, with continuous \emph{coefficient functions}, defined as follows:
$$\pi_{\xi,\eta}(x):G\to \C \;\;\;\;  x\mapsto \innprod{\pi(x)\xi,\eta}, \ \mbox{ for every } \xi,\eta \in \Hil.$$
The collection of all coefficient functions for (equivalence classes of) continuous unitary representations of $G$ is the Fourier-Stieltjes algebra $B(G)$. The space $B(G)$ forms a commutative Banach algebra under the pointwise operations and norm defined as 
$$\norm{\varphi}=\inf_{\varphi=\pi_{\xi,\eta}}\norm{\xi}\norm{\eta}.$$
Alternatively, $B(G)$ can be realized as $C^*(G)^*$, the continuous linear dual to the full group C*-algebra.

The Fourier algebra of a locally compact group $G$ is defined as the Banach subalgebra of $B(G)$ generated by compactly supported functions. It turns out that $A(G)$ is the set of all coefficient functions of the left regular representation of $G$ acting on the Hilbert space $L^2(G)$ under the inner product $\innprod{f,g}=\int_G f(x)\overline{g(x)}$ defined as follows:
$$L:G\to {\mathcal U}(L^2(G)), \ (L(x)f)(z)=f(x^{-1}z)$$
In \cite{Eymard}, Eymard first constructs $A(G)$ as the predual to the group von Neumann algebra, and then shows that these two definitions are equivalent  \cite[Chapter 5]{Eymard}. 
\end{rmk}
\begin{rmk}[Alternative approaches to the Fourier and Fourier-Stieltjes algebras]
    In \cite{Renault} Renault studies \emph{measurable} versions of the algebras $B(\G)$ and $A(\G)$ in the context of \emph{measurable} Hilbert fields. The definitions are identical except for the two continuity conditions replaced with measurability (e.g., the coefficients need only be measurable, not necessarily continuous). 

Paterson suggests three different definitions for the Fourier algebra in \cite{PatersonPaper}, each of which mimic one presentation of the algebra in the group case. All three definitions reduce to the existing Fourier algebra in the case when the groupoid is just a group, however it seems to be unknown whether the three definitions coincide for general groupoids. 
The first of the alternate definitions is to define $A(\G)$ to be the closure of coefficients coming from the set of sections which vanish at infinity. The second, due to \cite[pg 186]{Oty} suggests defining $A(\G)$ to be the closure of $C_c(\G)\cap B(\G)$ in the $B(\G)$ norm.
\end{rmk}
\subsection{Amenability of Banach algebras}
\label{sec:banachalgebrasandexamples}
We will assume familiarity with the ideas surrounding Banach algebras, Banach modules, and amenable groups. For any Banach space in this paper, we will take our scalar field to be $\C$. 
To denote the product of two elements $x,y$ in a Banach algebra, we use $x*y$ or $xy$ interchangeably. To show the module action, we use $a\cdot x$ or $x\cdot a.$
Every Banach $A$-bimodule $E$ leads to a \emph{dual $A$-bimodule} $E^*$ via the actions 
    $$(\varphi\cdot a)(x)=\varphi(a\cdot x) \ \mbox{ and }\ (a\cdot \varphi)(x)=\varphi(x\cdot a), \ \mbox{ for } x\in E, a\in A, \varphi\in E^*.$$

An \emph{approximate identity} of a Banach algebra $A$ is a net of elements $\set{\psi_\lambda}_{\lambda\in I}$ such that $\norm{\psi_\lambda * x-x}_A\to_\lambda 0$ and $\norm{x*\psi_\lambda-x}_A\to_\lambda 0$ for every $x\in A$. If $\sup\set{\norm{\psi_\lambda}_A:\lambda \in I}<\infty$ then the net is a \emph{bounded approximate identity}.

The idea of linking amenability of groups to the structure of their group algebra is what prompted the definition and study of the amenability of Banach algebras. Roughly speaking, an amenable Banach algebra is one that does not admit nontrivial derivations. 
A linear map $D$ from a Banach algebra $A$ to a Banach $A$-bimodule $E$ is called a \emph{derivation} if 
$$D(ab)=a\cdot D(b)+(D(a))\cdot b, \ \mbox{ for every } a,b\in A.$$
Easy examples of derivations are $D_a(x)=a\cdot x-x\cdot a$, for a fixed element $a\in E$; these are called \emph{inner} derivations. 

\begin{defn}\label{def:banachalgebraamenability}
    A Banach algebra $A$ is called: \begin{enumerate}
        \item \emph{amenable} if every bounded derivation from $A$ into a dual Banach $A$-bimodule is inner;
        \item \emph{weakly amenable} if every bounded derivation $D:A\to A^*$ is inner;
        \item \emph{contractible} if every bounded derivation $D:A\to E$, where $E$ is a Banach $A$-bimodule, is inner.
    \end{enumerate}
\end{defn}

\section{Restriction results on  \texorpdfstring{$A(\G)$}{A(G)} and  \texorpdfstring{$B(\G)$}{B(G)}}
\label{sec:generalrestrictionresults}
For a locally compact group $G$ and a closed subgroup $H$, the restriction map $A(G)\to A(H)$ is surjective \cite{Herz}. In the case of the Fourier-Stieltjes algebra, the restriction map $B(G)\to B(H)$ is well-defined, though it is not necessarily surjective (see, for example, \cite[pg. 92]{MahyaThesisPaper}). In this section, we focus on analogous results in the groupoid setting. 
Namely, we show that the restriction map $B(\G)\to \B(\G_u^u)$ is a continuous algebra homomorphism. When $\G$ is \'{e}tale, we show that the restriction map $A(\G)\to A(\G_u^u)$ is a continuous and surjective algebra homomorphism. Next, we demonstrate cases where the restriction map $B(\G)\to B(\G_u^u)$ is or is not surjective.

Let $\G$ be a locally compact Hausdorff groupoid. For every $u\in \unsp$, the space $\G_v^u$ equipped with the subspace topology is a locally compact Hausdorff space as well. In particular, $\G_u^u$ is a locally compact Hausdorff group for every $u\in \unsp$.
\begin{lem}\label{lem:measurelemma}
    Let $\G$ be a locally compact Hausdorff groupoid that admits a Haar system $\set{\lambda^u}_{u\in \unsp}$. For $u\in \unsp$, if $\lambda^u$ restricted to $\G_u^u$ is non-zero then it is a Haar measure for $\G_u^u$. 
\end{lem}
\begin{proof}
We denote $\lambda^u$ restricted to $\G_u^u$ by $\lambda_u^u$.
The left translation invariance of $\lambda_u^u$ follows directly from the analogous property of the Haar system $\{\lambda^u\}_{u\in\unsp}$, i.e., from \Cref{def:HaarSystem} (\ref{haarsysteminvariance}).
To prove the lemma, it remains to verify that $\lambda_u^u$ is both inner and outer regular, and finite on compact sets.

To show $\lambda_u^u$ is inner regular, consider an arbitrary open subset $O\subseteq \G_u^u$, and fix $\varepsilon>0$. Let $U$ be an open subset of $\G^u$ such that $O=U\cap \G_u^u$. As $\lambda^u$ is inner regular on $\G^u$, there is a compact set $C\subseteq U\subseteq \G^u$ with $\lambda^u(U\setminus C)\leq\varepsilon$. Given that $\G_u^u$ is closed in $\G^u$, we have $C\cap \G_u^u$ is compact as well. So,
    $$\lambda_u^u(O\setminus(C\cap \G_u^u))=\lambda_u^u((U\cap \G_u^u)\setminus(C\cap \G_u^u))=\lambda_u^u((U\setminus C)\cap\G_u^u)=\lambda^u((U\setminus C)\cap\G_u^u)\leq \varepsilon,$$
    which implies that $\lambda_u^u(O)\leq\lambda_u^u(C\cap \G_u^u)+\varepsilon$.
    So $\lambda_u^u$ is inner regular.

 For the outer regularity of $\lambda_u^u$, consider a Borel subset $E$ of $\G_u^u$, and fix $\varepsilon>0$.
 As $\lambda^u$ is an outer regular measure on $\G^u$, there is an open subset $O$ of $\G^u$ such that $\lambda^u(O)\leq \lambda^u(E)+\varepsilon$. Consider the set $V=O\cap \G_u^u$, which is open in $\G_u^u$. Clearly $E\subseteq V$, and we have
$$\lambda_u^u(V)=\lambda_u^u(O\cap \G_u^u)=\lambda^u(O\cap \G_u^u)\leq \lambda^u(O)\leq \lambda^u(E)+\varepsilon=\lambda_u^u(E)+\varepsilon.$$ 
Since $\varepsilon$ was arbitrary, we have that $\lambda_u^u$ is outer regular.

To prove that $\lambda_u^u$ is finite on compact sets, we first establish this property for $\lambda^u$. Let $K$ be a compact subset of $\G^u$ and fix $\varepsilon>0$. Since $\G^u$ is locally compact, and $\lambda^u$ is outer regular, there exists a precompact open subset $U$ of $\G^u$ containing $K$ with $\lambda^u(U\setminus K)<\varepsilon$. Next, we apply Urysohn's lemma (for locally compact Hausdorff spaces) to find a continuous and compactly supported function $f:\G^u\to [0,1]$ such that $f|_K=1$ and ${\rm supp}(f)\subseteq U$. By the Tietz extension theorem, we may extend this $f$ to an $\tilde{f}\in C_c^+(\G)$. By \Cref{def:HaarSystem} (\ref{haarsystemfunction}), the value $\int_{\G^u}\tilde{f}d\lambda^u$ is finite. This implies that $\lambda^u(K)$ must be finite, as we have
$$\lambda^u(K)=\int_K 1d\lambda^u\leq \int_{\G^u}\tilde{f}d\lambda^u=\tilde{f}^{0}(u)<\infty.$$
Finally, for every compact subset $C\subseteq \G_u^u$, we have
$\lambda_u^u(C)=\lambda^u(C)<\infty$, as compactness remains unchanged whether we consider $C$ as a subset of $\G^u$ or $\G^u_u$.

\end{proof}

\begin{defn}\label{def:condition*}
    A groupoid $\G$ with a Haar system $\set{\lambda^u}_{u\in \unsp}$ is said to \emph{satisfy Condition $(*)$} if for every $u\in \unsp$ the restriction of $\lambda^u$ to $\G_u^u$ is non-zero.
\end{defn}

 \begin{rmk}\label{rmk:condition*remark}
     \begin{enumerate}
         \item We require Condition $(*)$ to ensure the non-triviality of the spaces $L^2(\G_u^u,\lambda^u|_{\G_u^u})$. 
         \item Consider $X:=[0,1]$ with the Lebesgue measure, and set $\G:=X\times X$ to be the full equivalence relation groupoid on $X$. In this groupoid, $\G^{(x,x)}:=\G^x=\set{(x,y):y\in [0,1]}$. Each measure $\lambda^x$ on $\G^{x}$ will be just the Lebesgue measure. Then, for $x\in X$ we have $\G_x^x=\set{(x,x)}$ and $\lambda^{x}(\G_x^x)=0$, so $\lambda^x$ does not restrict to the Haar measure of $\G_x^x$.
         \item Discrete groupoids, group bundles with open range map, and \'{e}tale groupoids all satisfy Condition $(*)$. \cite[Lemma 1.3]{ToPDown} shows that group bundles with open range map have Haar systems, and \cite[Theorem 6.9]{ToolKit} demonstrates that any Haar system of such group bundles is essentially comprised of the Haar measures of the groups.
     \end{enumerate}
 \end{rmk}

\begin{prop}\label{thm:RestrictionTheorem}
        Let $\G$ be a locally compact, Hausdorff groupoid, and fix $u\in \unsp$. 
        \begin{enumerate}
            \item\label{item-0-codomain} Let $\fH$ be a continuous Hilbert field over the topological space $\unsp$, and $\calpi$ be a representation of $\G$ on $\fH$. Then the map 
            $$\G_u^u\to \U(\fH_u), \ x\mapsto \calpi(x)$$ 
            is a continuous unitary representation of $\G_u^u$. 
            \item\label{item-1-thm1} The map $R_{B(\G)}: B(\G)\to B(\G_u^u)$, $\varphi\mapsto \varphi|_{\G_u^u}$ is a contractive algebra homomorphism.
            \item\label{item-2-thm1} Let $\RA$ denote the map $R_{B(\G)}$ restricted to $A(\G)$. Then $\RA:A(\G)\to A(\G_u^u)$ is a contractive algebra homomorphism.
        \end{enumerate}
    \end{prop}
\begin{rmk}
    The restriction maps $R_{B(\G)}$ and $\RA$ depend on $u$. To avoid the clutter of notation, we have chosen not to include $u$ in the name of these maps except when necessary.
\end{rmk}
\begin{proof}
       Clearly, the mapping $x\mapsto \calpi(x)$ is a group homomorphism $\G_u^u\to \U(\fH_u)$, which we denote by $\pi$. Next, we show that the map $x\mapsto \pi_{\xi,\eta}(x)$ is continuous for every $\xi,\eta\in \fH_u$. To do this, fix $\xi,\eta\in \fH_u$. By \cite[Proposition 10.1.10]{Dixmier}, there exist bounded continuous sections $\bxi,\boldeta$ such that $\bxi(u)=\xi$ and $\boldeta(u)=\eta$. By \Cref{def:ghilbdle}, the coefficient function $x\mapsto \fH_{\bxi,\boldeta}(x)$ is continuous on $\G$, thus its restriction to $\G_u^u$ will be continuous. It is easy to see that $\fH_{\bxi,\boldeta}|_{\G_u^u}=\pi_{\xi,\eta}$, proving $\pi_{\xi,\eta}$ is continuous on $\G^u_u$.

       By \eqref{item-0-codomain}, the domain and codomain of $R_{B(\G)}$ are as claimed. As the algebra operations in both domain and codomain are pointwise, the mapping will be an algebra homomorphism. 
       Next, we show that $R_{B(\G)}$ is contractive. Consider an element $\varphi\in B(\G)$ represented as $\varphi=\fH_{\bxi,\boldeta}$. Letting $\xi=\bxi(u)$ and $\eta=\boldeta(u)$, we have $\varphi|_{\G_u^u}=\pi_{\xi,\eta}$. So, we have
       $$\norm{R_{B(\G)}(\varphi)}=\norm{\pi_{\xi,\eta}}_{B(\G_u^u)}\leq \norm{\xi}\norm{\eta}\leq \left(\sup_{u\in \unsp}\norm{\bxi(u)}\right)\left(\sup_{u\in \unsp}\norm{\boldeta(u)}\right)\leq \|\bxi\|_\Delta\|\boldeta\|_\Delta.$$
       Since this inequality is true for arbitrary representations of $\varphi$ as coefficient functions, by the definition of the norm of $B(\G)$ we get
       $\norm{R_{B(\G)}(\varphi)}_{B(\G_u^u)}\leq \norm{\varphi}_{B(\G)}$, so $R_{B(\G)}$ is contractive.

       To prove \eqref{item-2-thm1}, we will first show that the map is well-defined, i.e., the image of $\RA$ is contained in $A(\G_u^u)$.
       It will then follow that $\RA:A(\G)\to A(\G_u^u)$ is a contractive algebra homomorphism.
        Recall that $A(\G)=\overline{A_c(\G)}^{\norm{\cdot}_{B(\G)}}$, and $A_c(\G)$ is a subalgebra of $C_c(\G)$. By \eqref{item-0-codomain}, we have
       $$R_{B(\G)}(A_c(\G))\subseteq R_{B(\G)}(C_c(\G)\cap B(\G))
       \subseteq C_c(\G_u^u)\cap B(\G_u^u)\subseteq A(\G_u^u),$$ 
       where, in the second inclusion, we used the fact that the restriction map sends elements of $C_c(\G)$ into $C_c(\G^u_u)$, since $\G_u^u$ is closed in $\G$. 
       Moreover, since $R_{B(\G)}$ is continuous, we have 
       $R_{B(\G)}(A(\G))\subseteq
       \overline{R_{B(\G)}(A_c(\G))}^{\norm{\cdot}_{B(\G)}}\subseteq A(\G_u^u).$
       So, $\RA$ is well-defined.
\end{proof}

To show surjectivity of $\RA$, we need to use the coefficient space of a particular representation. 
\begin{defn}\cite[Definition 1.5]{Arsac}\label{def:fouriercoefficientspace}
   For a continuous unitary representation $\pi$ of a locally compact group $G$, the coefficient space of $\pi$ is defined as 
   $$A_\pi(G):=\overline{\text{span}\set{\pi_{\xi,\eta}:\xi,\eta \in \Hil_\pi}},$$
   where the closure is taken with respect to the norm of $B(G)$.
\end{defn}

\begin{thm}\label{thm:denserange}
  Let $\G$ be a locally compact Hausdorff \'etale groupoid, and fix $u\in \unsp$. 
  The restriction map $\RA: A(\G)\to A(\G_u^u)$ is surjective.
\end{thm}
\begin{proof}
Let $\Lambda$ be the left regular bundle with linear actions $\set{L_x}_{x\in \G}$ on the continuous Hilbert field $\{L^2(\G^v)\}_{v\in\unsp}$. 
By \Cref{thm:RestrictionTheorem} \eqref{item-0-codomain}, the restriction of $\Lambda$ to $\G_u^u$ is a continuous unitary representation, which we denote by $\sigma$. Namely, 
$$\sigma:\G_u^u\to \U(L^2(\G^u)), \ (\sigma(x)f)(z)=f(x^{-1}z).$$ 
Note that $\RA(\Lambda_{F,G})=\sigma_{F^u,G^u}$. This fact implies that 
$\RA(A(\G))\subseteq A_\sigma(\G_u^u)$, since $\RA$ is a continuous algebra homomorphism. 

\begin{claim}\label{claim:claim1}
   $\RA:A(\G)\to A_\sigma(\G_u^u)$ is surjective. 
\end{claim}
\noindent \emph{Proof of \Cref{claim:claim1}:}
First, consider a coefficient function $\sigma_{f,g}$ of $\sigma$ associated with $f,g\in L^2(\G^u)$. By \cite[Proposition 10.1.10]{Dixmier}, there exist continuous sections $F,G$ of $\Lambda$ such that $F^u=f$, $G^u=g$, $\norm{F}_\Delta=\norm{f}_2$, and $\norm{G}_\Delta=\norm{g}_2$. It is easy to see that 
$\RA(\Lambda_{F,G})=\sigma_{f,g}$, i.e., $\sigma_{f,g}$ belongs to the range of $\RA$. The linearity of $\RA$ establishes that 
$$\text{span}\set{\sigma_{f,g}:f,g\in L^2(\G^u)}\subseteq \RA(A(\G)).$$

    Next, let $\varphi \in A_\sigma(\G_u^u)$ be arbitrary. Via \cite[Theorem 2.8.4]{Kaniuth}, there exist
    $\{f_n\}_{n\in {\mathbb N}},\{g_n\}_{n\in {\mathbb N}}\subseteq L^2(\G^u)$ satisfying the following:
    $$\varphi=\sum_{n=1}^\infty \sigma_{f_n,g_n}, \ \mbox{ and } \norm{\varphi}_{A_\sigma(\G_u^u)}=\sum_{n=1}^\infty\norm{f}_2\norm{g}_2.$$
    For each $\sigma_{f_n,g_n}$, we take the preimage $\Lambda_{F_n,G_n}$ as constructed above, and observe that
    $$\norm{\Lambda_{F_n,G_n}}_{A(\G)}\leq \norm{F_n}_\Delta \norm{G_n}_\Delta=\norm{f_n}_2\norm{g_n}_2.$$
    So, $\sum_{n=1}^\infty \norm{\Lambda_{F_n,G_n}}_{A(\G)}$ is convergent,  which implies that $\sum_{n=1}^\infty \Lambda_{F_n,G_n}$ converges in $A(\G)$. 
    Now by continuity of $\RA$, we get
    \begin{align*}
        R_{A(\G)}\of{\sum_{n=1}^\infty \Lambda_{F_n,G_n}}=
       \sum_{n=1}^\infty\sigma_{f_n,g_n}=\varphi,
    \end{align*}
  which implies that $A_\sigma(\G_u^u)\subseteq \RA(A(\G))$. 

To finish the proof, we only need to show that $A_\sigma(\G^u_u)=A(\G_u^u)$.
\begin{claim}\label{claim:claim2}
Let $L_{\G_u^u}:\G_u^u\to {\mathcal U}(L^2(\G_u^u))$ be the left regular representation of $\G_u^u$. We have the following unitary equivalence of representations:
    $$\sigma\simeq\oplus_{v:\G_v^u\neq \varnothing}L_{\G_u^u}.$$ 
\end{claim}
\noindent \emph{Proof of \Cref{claim:claim2}:}
Given that $\G$ is \'{e}tale, we have that $\G^u$ is discrete, and the Hilbert space $L^2(\G^u)$ can be written as an orthogonal direct sum $\oplus_{v:\G_v^u\neq \varnothing}L^2(\G_v^u)$.
Fix $v\in \unsp$ with $\G_v^u \neq \varnothing$. 
It is easy to see that $L^2(\G_v^u)$ is a closed, nonzero $\sigma$-invariant subspace of $L^2(\G^u)$.  Let $\sigma^{L^2(\G_v^u)}$ denote the restriction of $\sigma$ to this invariant subspace. 
Recall that when $\G$ is \'{e}tale, we fix the Haar system $\{\lambda^v\}_{v\in\unsp}$ for $\G$ to be just the collection of counting measures (\Cref{assump-etale}). So, by \Cref{lem:measurelemma}, the restriction of $\lambda^u$ to $\G_u^u$ is the Haar measure of $\G_u^u$.\footnote{Formally speaking, when $\G$ is \'etale, it satisfies Condition $(*)$; see \Cref{def:condition*}.} 
Consequently, the representation $\sigma^{L^2(\G_u^u)}$ is exactly the left regular representation of $\G_u^u$. 

We further claim that $\sigma^{L^2(\G_v^u)}\simeq L_{\G_u^u}$ for each $v$ with $\G_v^u\neq \varnothing$. Indeed, for each such $v$, fix $\gamma_v\in \G_v^u$, and define 
$$U_v:L^2(\G_v^u)\to L^2(\G_u^u), \ (U_vf)(z)=f(z\gamma_v).$$ 
Clearly $U_v$ is invertible. Moreover, since the function $\G_u^u\to \G_v^u$, $z\mapsto z\gamma_v$ is bijective, we have
$$\norm{U_vf}_2^2=\sum_{z\in \G_u^u}\abs{f(z\gamma_v)}^2=\sum_{z\in \G_v^u}\abs{f(z)}^2=\norm{f}_2^2.$$
This implies that $U_v$ is a surjective isometry, and thus, a unitary map. It is easy to verify that $(U_v \sigma^{L^2(\G_v^u)}_x U_v^{-1})(f)=\sigma^{L^2(\G_u^u)}_x f$, that is, the unitary operator $U_v$ intertwines $\sigma^{L^2(\G_v^u)}$ and $\sigma^{L^2(\G_u^u)}$.
Finally, since the spaces $L^2(\G_v^u)$ are orthogonal for distinct $v$, we get
$$\sigma\simeq\bigoplus_{v:\G_v^u\neq \varnothing}\sigma^{L^2(\G_v^u)}\simeq \bigoplus_{v:\G_v^u\neq \varnothing} L_{\G_u^u}.$$

So the two representations $\sigma$ and $L_{\G_u^u}$ are quasi-equivalent (by \cite[Proposition 5.3.1]{Dixmier}), which implies that $A_\sigma(\G_u^u)=A(\G_u^u)$
(by \cite[Proposition 2.8.12]{Kaniuth}).  
\end{proof}

\begin{rmk}
The condition that $\G$ is \'{e}tale is not strictly necessary, though we do not know if \Cref{thm:denserange} is true in full generality. We explicitly use the \'etale condition for two reasons. Firstly, to make sure that $\G$ satisfies Condition ($*$) in \Cref{def:condition*}; and secondly, to guarantee that the maps $U_v$ are isometric. In the latter consideration, one needs the measure $\lambda^u$ of the groupoid's Haar system to be both left and right invariant, that is, in addition to \Cref{def:HaarSystem} \eqref{haarsysteminvariance}, we need to have
$$\int_{\G^u}f(xz)\lambda^u(x)=\int_{\G^u}f(x)d\lambda^u(x)$$ 
for $f\in L^2(\G^u)$. This is certainly true for \'{e}tale groupoids, since the collection of counting measures is the chosen Haar system in this scenario. This condition will also be true for group bundles which have a Haar system for all cases of $u\in \unsp$ such that $\G^u$ is a unimodular group.
\end{rmk}

We end this section with an application of \Cref{thm:denserange} in understanding Banach algebra features of the Fourier algebra. A property $P$ of a Banach algebra ${\mathcal A}$ will be called \emph{hereditary} if it is preserved under continuous homomorphisms with dense range, i.e., if ${\mathcal A}$ has property $P$, and $\theta:{\mathcal A}\to {\mathcal B}$ is a continuous algebra homomorphism such that $\overline{\theta({\mathcal A})}={\mathcal B}$, then ${\mathcal B}$ has property $P$ as well. An easy example of hereditary properties is the existence of a bounded approximate identity in a Banach algebra.
Contractibility and amenability of Banach algebras, and weak amenability of commutative Banach algebras are all hereditary properties (see \cite[Proposition 2.8.64]{Dales}, \cite{BadeCurtisDalesPaper}, and \cite[Proposition 2.3.1]{Runde}).

\begin{cor}\label{cor:hereditary}
      Let $\G$ be a locally compact, Hausdorff \'{e}tale groupoid. Let $u\in \unsp$.
      \begin{itemize}
      \item[(i)] If $\G_u^u$ is not amenable, then $A(\G)$ does not contain a bounded approximate identity.
      \item[(ii)] If $A(\G_u^u)$ is not amenable (resp.~weakly amenable, resp.~contractible), then $A(\G)$ is not amenable (resp.~weakly amenable, resp.~contractible).
      \item[(iii)] If $\G_u^u$ contains a closed subgroup $H$ such that $A(H)$ is not (weakly) amenable, then $A(\G)$ is not (weakly) amenable.
      \end{itemize}
 \end{cor}

\begin{proof}
By \Cref{thm:denserange}, the restriction map $\RA:A(\G)\to A(\G_u^u)$ is a continuous surjective algebra homomorphism when $\G$ is \'{e}tale. So it preserves hereditary properties such as contractibility, amenability and weak amenability (of commutative Banach algebras). This proves (ii).

To prove (i), note that by Leptin's theorem, when $\G_u^u$ is not amenable, the Fourier algebra $A(\G_u^u)$ does not admit a bounded approximate identity \cite{Leptin}. 
As the existence of a bounded approximate identity is hereditary, and $\RA:A(\G)\to A(\G_u^u)$ is a surjective algebra homomorphism, $A(\G)$ does not admit a bounded approximate identity as well.

Next, note that the restriction map $r_H:A(\G_u^u)\to A(H)$ is always surjective \cite{Herz}. So the composition $r_H\circ\RA: A(\G)\to A(H)$ is also a surjective algebra homomorphism, and thus, it preserves hereditary properties. So, the amenability or weak amenability of $A(\G)$ implies the corresponding property for $A(H)$. This proves (iii).
\end{proof}

\section{Surjectivity of \texorpdfstring{$R_{B(\G)}$}{R(B(G))} for certain classes of groupoids}
\label{sec:specificrestrictionresults}
Let $\G$ be a locally compact Hausdorff groupoid. In this section, we show that the restriction map $R_{B(\G)}$ is surjective if $\G$ is transitive with a continuous section (\Cref{thm:transitivesurjection}) or $\G$ is a group bundle with discrete unit space (\Cref{prop:bundlesurjection}). On the other hand, we discuss necessary conditions for surjectivity of $R_{B(\G)}$, when $\G$ is an HLS groupoid (\Cref{thm:HLS-not-surjective}), and present examples where the restriction map is not onto. 
\subsection{Transitive groupoids}
\label{sec:transitiverestrictions}
A groupoid is called \emph{transitive} if $\G_u^v\neq \varnothing$ for every pair $u,v\in \unsp$. In such groupoids, the restricted range map $r:\G_u\to \unsp$ is surjective for any choice of $u\in\unsp$. When $\G$ is a transitive groupoid, evaluation of coefficient functions of $\G$-Hilbert bundles will only depend on the actions of a subset of $\G$.

Let $\G$ be a transitive locally compact Hausdorff groupoid. Fix $u\in \unsp$. For each $v\in \unsp\setminus \{u\}$, fix an element $\gamma_v\in \G_u^v$. We set $\gamma_u:=u$. For every pair $v,w\in\unsp$, define the following map:
\begin{equation}\label{eq:fibermappings}
      \varphi_{v,w}:\G_u^u\to \G_v^w, \ \varphi_{v,w}(x)=\gamma_w x \gamma_v^{-1}. 
\end{equation}
By continuity of multiplication and inversion operations in $\G$, each map $\varphi_{v,w}$ is a homeomorphism from $\G_u^u$ to $\G_w^v$, and each map $\varphi_{v,v}$ is an isomorphism (of locally compact groups) from $\G_u^u$ to $\G_v^v$.

\begin{thm}\label{thm:transitivesurjection}
    Let $\G$ be a locally compact Hausdorff transitive groupoid. Suppose $\gamma:\unsp \to \G_u$ is a continuous section\footnote{It is an unfortunate coincidence that the terminology for this function matches that of $\G$-Hilbert bundle sections.} for the range map $r:\G_u\to \unsp$, i.e., $\gamma$ is continuous and $r\circ \gamma:\unsp \to \unsp$ is the identity map. Then for every $u\in \unsp$ the map $R_{B(\G)}: B(\G)\to B(\G_u^u)$ of \Cref{thm:RestrictionTheorem} is surjective. In particular $R_{B(\G)}:B(\G)\to B(\G_u^u)$ is surjective when $\unsp$ is discrete.
\end{thm}

\begin{proof}
Let $\fH$ be a $\G$-Hilbert bundle with linear actions $\set{\calpi(x)}_{x\in \G}$. Then for every $\bxi,\boldeta\in \Delta_b(\fH)$, the coefficient function $\fH_{\bxi,\boldeta}$ is determined by evaluations on $\G_u^u$ and $\set{\gamma_v}_{v\in \unsp}$ via
    $$\fH_{\bxi,\boldeta}(x)=\innprod{\calpi(\gamma_w^{-1}x\gamma_v)\\\calpi(\gamma_v^{-1})\bxi(s(x)),\calpi(\gamma_w)^*\boldeta(r(x))}$$
    where $w=r(x)$ and $v=s(x)$.

    Now, let $f\in B(\G_u^u)$, and suppose $f=\pi_{\xi_0,\eta_0}$ for some continuous unitary representation $\pi$ of $\G_u^u$. 
    To obtain a preimage of $f$, we first consider the constant Hilbert field $\fH$ with fibers $\fH_v:=\Hil_\pi$ for each $v\in \unsp$.
For $v\in \unsp$ set $\gamma_v=\gamma(v)$, and define $\varphi_{v,w}(x):=\gamma_w x\gamma_v^{-1}$ according to \Cref{eq:fibermappings}. Define $\calpi(x):=\pi(\varphi^{-1}_{s(x),r(x)}(x))$. 
It is easy to see that the map $x\mapsto \calpi(x)$ is a groupoid homomorphism.

As the Hilbert field of $\fH$ is a constant field, the continuous sections are just functions belonging to $C(\unsp,\Hil_\pi)$. 
Now, we show the continuity of the coefficients $\fH_{\bxi,\boldeta}$ for continuous sections $\bxi,\boldeta$. Let $\set{x_\alpha}_{\alpha\in I}$ be a net converging to $x$ in $\G$. 
Since the range and source maps $r,s:\G\to \unsp$ and the section map $\gamma:\unsp\to \G_u$ are continuous, we have that $\gamma_{r(x_\alpha)}\to \gamma_{r(x)}$ and $\gamma_{s(x_\alpha)}\to\gamma_{s(x)}$. So, we get
\begin{equation}\label{eq:conv-1}
\lim_{x_\alpha\to x}\varphi^{-1}_{s(x_\alpha),r(x_\alpha)}(x_\alpha)=\lim_{x_\alpha\to x}\gamma_{r(x_\alpha)}^{-1}x_\alpha \gamma_{s(x_\alpha)}=\gamma_{r(x)}^{-1}x \gamma_{s(x)}=\varphi^{-1}_{s(x),r(x)}(x).
\end{equation}
The continuity of $\bxi$ and $\boldeta$, together with \eqref{eq:conv-1}, gives that 
$$\lim_{x_\alpha\to x}\langle\pi(\varphi^{-1}_{s(x_\alpha),r(x_\alpha)}(x_\alpha))\bxi(s(x_\alpha)),\boldeta(r(x_\alpha))\rangle=\langle\pi(\varphi^{-1}_{s(x),r(x)}(x))\bxi(s(x)),\boldeta(r(x))\rangle,$$
since the unitary representation $\pi$ is WOT-continuous. This shows that $\fH_{\bxi,\boldeta}$ is continuous, showing that $\fH$ is a $\G$-Hilbert bundle. Consider the continuous sections $\bxi, \boldeta\in \Delta_b(\fH)$ defined as $\bxi(v)=\xi_0$ and $\boldeta(v)=\eta_0$ for every $v\in \unsp$. These continuous sections lead to the desired preimage as follows. Namely, for $x\in \G_u^u$, we have
    $$\fH_{\bxi,\boldeta}(x)=\innprod{\calpi(x)\xi_0,\eta_0}=\innprod{\pi(x)\xi_0,\eta_0}=\pi_{\xi_0,\eta_0}(x)=f.$$
\end{proof}

\subsection{Group bundles with discrete unit spaces}
\label{sec:discreteunitspacerestriction}
Let $X$ be an arbitrary set. For each $x\in X$, let $G_x$ be a locally compact Hausdorff group. Consider the group bundle formed by the collection $\set{G_x}_{x\in X}$ as defined in \Cref{exa:groupoidexample} (iii). This is the groupoid $\G=\cup_{x\in X}\set{x}\times G_x$ with a partial product $(x,g)(x,h)=(x,gh)$ and an inversion $(x,g)^{-1}=(x,g^{-1})$. As each $G_x$ embeds into $\G$, we simply regard $\G$ as the disjoint union of groups, that is, if $\zeta=(x,g)\in \G$ we simply write $\zeta\in G_x$. In addition, we identify the unit space of $\G$ with $X$ through $(x,e_{G_x})\mapsto x$. 

In this section, we assume that the group bundle $\G$ is equipped with a locally compact Hausdorff topology, and $\unsp$ is discrete. 
For every $x\in X$, let $\lambda^x$ denote the left Haar measure on $G_x$. It is easy to verify that the collection $\{\lambda^x\}_{x\in X}$ forms a Haar system on $\G$.

\begin{prop}\label{prop:bundlesurjection}
    Let $\G=\sqcup_{x\in X}G_x$ be a locally compact and Hausdorff group bundle with discrete unit space. For $u\in X$, let $R_{B(\G)}:B(\G)\to B(G_u)$ be the restriction map of \Cref{thm:RestrictionTheorem}. Then $R_{B(\G)}$ is surjective.
\end{prop}

\begin{proof}
    Let $\pi$ be a unitary representation of $G_u$ and $\xi,\eta$ be fixed vectors in $\Hil_{\pi}$. We will show that $\pi_{\xi,\eta}\in R_{B(\G)}(B(\G))$. To this end, we construct a continuous field of Hilbert spaces $\fH$ and a groupoid representation $\calpi$ of $\G$ on $\fH$ as follows.
    
     First, we define the Hilbert bundle $\fH=\{\fH_z\}_{z\in X}$ by setting $\fH_u=\Hil_{\pi}$ and $\fH_z:={\mathbb C}$ for each $z\neq u$. Since $X$ is discrete, $\fH$ together with $\Gamma=\prod_{x\in X}\fH_x$ becomes a continuous Hilbert field (see \Cref{exp-fields} \eqref{exp-discrete}). 
     Next, we consider the trivial representation $\pi^x:G_x\to {\mathbb T}$, for $x\neq u$,  mapping every group element to 1. We set $\pi^u=\pi$. We can now define the following representation of $\G$:
     \begin{eqnarray}
         \calpi((z,g_z)):=\pi^{z}(g_z)=\left\{
         \begin{array}{ll}
             1 & \mbox{ when } z\neq u\\
             \pi(g_z) & \mbox{ when } z=u
         \end{array}\right..
     \end{eqnarray}
    Clearly, $\calpi$ is a groupoid homomorphism. Next, we show that each coefficient function of $\calpi$ is continuous. Indeed, let ${\bxi,\boldeta}$ be continuous sections of $\fH$, and consider $\calpi_{\bxi,\boldeta}$. Let $\set{x_\alpha}_{\alpha\in I}$ be a net that converges to $x$ in $\G$. By continuity of the range map, $\set{r(x_\alpha)}_{\alpha\in I}$ converges to $r(x)$. This implies that $\set{r(x_\alpha)}_{\alpha\in I}$ is eventually constant, as $\unsp$ is discrete. So the continuity of $\calpi$ follows from the continuity of coefficient functions of each representation $\pi^z$. 

    Finally, let $\bxi$ and $\boldeta$ be continuous bounded sections satisfying $\bxi(u)=\xi$ and $\boldeta(u)=\eta$. (Note that since $X$ is discrete, any section will be a continuous section).  
    Then, we have $R_{B(\G)}(\calpi_{\bxi,\boldeta})=\pi_{\xi,\eta}$, and the restriciton map is surjective.
\end{proof}

\subsection{HLS groupoids}\label{sec:HLSGroupoids}
Next, we study a particular class of group bundles called the HLS\footnote{HLS stands for Higson, Lafforgue, and Skandalis, the authors who originated the construction.} groupoids, originally appearing in \cite{Skandalis}. 
Let $G$ be a discrete group. An \emph{approximating sequence} for $G$ is a sequence $\set{K_n}_{n\in \N}$ of normal subgroups of finite index in $G$ such that $K_n \supseteq K_{n+1}$ for all $n$, and $\bigcap_n K_n=\set{e_G}$. 
For every $n\in\N$, let $G_n$ denote the finite quotient group $G/K_n$, and $q_n: G\to G_n$ be the canonical quotient map. We set $G_\infty:=G$ and $q_\infty$ to be the identity map. We define $\N_*=\N\cup \set{\infty}$ to be the one point compactification of $\N$.

\begin{defn}\cite[Definition 2.1]{Willett}\label{def:hlsgroupoid}
Let $G$ be a discrete group with an approximating sequence $\{K_n\}_{n\in \N}$.
The HLS groupoid associated to $(G,\set{K_n}_{n\in \N})$ is the group bundle
    $$\G:=\sqcup_{n\in \N_*}\set{n}\times G_n,$$
    equipped with the topology generated by the following open sets:
    \begin{enumerate}
        \item the singleton $\set{(n,g)}$, for each $n\in \N$ and $g\in G_n$;
        \item the set $\ray_n(g):=\set{(m,\pi_m(g)): m\in \N_*, m\geq n}$ for each $g\in G$ and $n\in \N$.
    \end{enumerate}
\end{defn}
We refer to sets of the second type as \emph{rays}, although \cite{Alekseev} refers to these sets as tails. It is easy to see that every ray in an HLS groupoid is open, closed, and compact. Consequently, HLS groupoids are locally compact and Hausdorff.

The main results of this section are based on the concept of weak containment, which we recall now.
Let $S$ and $T$ be sets of continuous unitary representations of a locally compact group $G$. Every continuous unitary representation $\pi$ of $G$ extends to a non-degenerate $*$-representation of the group C*-algebra, namely $C^*(G)$, which we denote by $\pi$ again. Let $\ker_{C^*}(\pi)$ denote the kernel of the extended $*$-representation $\pi$. We say \emph{$S$ is weakly contained in $T$}, denoted $S\prec T$, if $\cap_{\tau\in T} \ker_{C^*}(\tau) \subseteq \cap_{\sigma\in S} \ker_{C^*}(\sigma)$ (see e.g., \cite[Definition 1.7.2]{Kaniuth}).

\begin{prop}\label{prop:HLSweakcontainment}
Let $\G$ be the HLS groupoid associated to a discrete group $G$ with an approximating sequence $\set{K_n}_{n\in \N}$.
Let $\fH$ be a continuous Hilbert field, and $\calpi$ be a representation of $\G$ on $\fH$. For $n\in \N_*$, we denote by $\pi^n$ the representation obtained from restricting $\calpi$ to the fiber $G_n$ (\Cref{thm:RestrictionTheorem}~\eqref{item-0-codomain}). Then, the representation $\pi^\infty$ is weakly contained in the set of representations $\set{\pi^n\circ q_n :n\in \N}$.
\end{prop}

\begin{proof}
    For each $n\in\N_*$, let $\Hil_{\pi^n}$ denote the associated Hilbert space.
    We proceed via \cite[Proposition 1.7.6]{Kaniuth} whereby we will demonstrate that for any vector $\zeta\in \Hil_{\pi^\infty}$, there is a sequence $\set{\psi_n}_{n\in \N}$ in $B(G_\infty)$ such that:
    \begin{enumerate}
        \item\label{item:weak-lem-1} $\psi_n\to \pi^{\infty}_{\zeta,\zeta}$ uniformly on all compact subsets $K\subseteq G_\infty$;
        \item\label{item:weak-lem-2} Each $\psi_m$ is a linear combination of elements of $\set{(\pi^n\circ q_n)_{\xi_n,\xi_n}:\xi_n\in \Hil_{\pi^n}, n\in \N}$.
    \end{enumerate}
Since every compact subset of the discrete group $G_\infty$ is finite, we can update \eqref{item:weak-lem-1} and require a sequence $\{\psi_n\}$ that converges pointwise to $\pi^\infty_{\zeta,\zeta}$.

    By \cite[Proposition 10.1.10]{Dixmier}, there exists a continuous section $\bxi\in \Delta_b(\fH)$ with $\bxi(\infty)=\zeta$. 
    For $x\in G_m$, we define $$\psi_m(x):=\innprod{\pi^m(q_m(x))\bxi(s(x)),\bxi(r(x))}.$$ 
    Clearly, $\psi_n(x)=(\pi^n\circ q_n)_{\bxi(n),\bxi(n)}$ for every $n\in\N$.
    Now, fix $z\in G_\infty$, and note that $\set{(n,q_n(z))}_{n\in \N_*}$ converges to $(\infty,z)$ in $\G$. Since the coefficient function $\fH_{\bxi,\bxi}$ is continuous at $(\infty,z)$, we have 
    \begin{eqnarray*}
    \lim_{n\to \infty}\psi_n(z)&=&\lim_{n\to \infty}\innprod{\pi^n(q_n(z))\bxi(n),\bxi(n)}\\
    &=&\lim_{n\to \infty}\calpi_{\bxi,\bxi}((n,q_n(z))\\
    &=&\calpi_{\bxi,\bxi}((\infty,z))\\
    &=&\innprod{\pi^\infty(z)\bxi(\infty),\bxi(\infty)}
=\pi^{\infty}_{\zeta,\zeta}(z).
    \end{eqnarray*}
    This finishes the proof.
\end{proof}
\begin{rmk}\label{rmk-admitting-approx-seq}
Every residually finite and finitely generated group admits an approximating sequence; see \cite[Lemma 2.8]{Willett} for a proof for $\F_2$, and \cite{DG-HLS} for a straightforward generalization to all residually finite and finitely generated groups.
\end{rmk}

Let $G$ be a residually finite group. Set $\widetilde{G}$ to denote the collection of (equivalence classes of) continuous unitary representations of $G$ and $\widetilde{G}_F$ the subset of representations that factor through a finite quotient. A group will \emph{have property FD} if for every $\pi \in \widetilde{G}$, $\pi$ is weakly contained in the set $\widetilde{G}_F$ (see \cite[pg. 2]{Lubotzky}). Equivalently, $G$ has property FD if $\widetilde{G}_F$ is dense in $\widetilde{G}$ in the Fell topology. 

\begin{thm}\label{thm:HLS-not-surjective}
    Let $\G$ be the HLS groupoid associated to a discrete group $G$ with an approximating sequence $\set{K_n}_{n\in \N}$.
    Consider the restriction map $R_{B(\G)}:B(\G)\to B(G_\infty)$ of \Cref{thm:RestrictionTheorem}. If $R_{B(\G)}$ is surjective, then $G=G_\infty$ has property FD. 
\end{thm}
\begin{proof}
Recall that $\widetilde{G_n}$ and $\widetilde{G}$ denote the collection of continuous unitary representations of $G_n$ and $G$ respectively.
Suppose  $R_{B(\G)}$ is surjective, and let $\pi$ be an arbitrary unitary representation of $G$. We will show that $\pi$ is weakly contained in 
$${\mathcal F}:=\bigcup_{n\in\N}\set{\pi\circ q_n:\pi \in \widetilde{G}_n}.$$ 
To see this, let $\xi,\eta\in \Hil_\pi$ be arbitrary, and consider the coefficient function $\pi_{\xi,\eta}\in B(G)$. As $R_{B(\G)}$ is surjective, there is a groupoid representation $\calpi$ on a continuous Hilbert field $\fH$, and continuous sections $\bxi,\boldeta\in \Delta_b(\fH)$ such that 
$$R_{B(\G)}(\fH_{\bxi,\boldeta})=\pi_{\xi,\eta}.$$
Let $\pi^\infty$ denote the representation formed by restricting $\calpi$ to $G_\infty$ (as described in \Cref{thm:RestrictionTheorem} \eqref{item-0-codomain}). Clearly, 
$$R_{B(\G)}(\fH_{\bxi,\boldeta})=\pi^\infty_{\bxi(\infty),\boldeta(\infty)}=\pi_{\xi,\eta}.$$ By \Cref{prop:HLSweakcontainment}, $\pi^\infty$ is weakly contained in ${\mathcal F}$. Therefore, $\pi^\infty_{\bxi(\infty),\boldeta(\infty)}$ 
can be written as uniform limit on compact sets of linear combinations of coefficient functions of representations in ${\mathcal F}$. So, the same holds for $\pi_{\xi,\eta}$, and $\pi$ is weakly contained in ${\mathcal F}$ as well. 
\end{proof}

\begin{cor}\label{cor:surjectivecounterexample}
    There exist HLS groupoids for which $R^\infty_{B(\G)}$ is not surjective.
\end{cor}

\begin{proof}
Let $\G$ be an HLS groupoid associated with a discrete group $G$ with an approximating sequence $\set{K_n}_{n\in \N}$. If $R^\infty_{B(\G)}$ is surjective,
then by \Cref{thm:HLS-not-surjective}, the group $G$ must have property FD.
Thus, to prove the corollary, it is enough to find a discrete group $G$ that does not satisfy property FD but admits an approximating sequence. By \Cref{rmk-admitting-approx-seq}, it is enough to find a finitely generated and residually finite discrete group $G$ that does not satisfy property FD.
For a specific example of such a group $G$, take $G={\rm SL}_n({\mathbb Z})$ with $n\geq 3$. Clearly, $G$ is a finitely generated linear group over a field of characteristic zero. 
By Mal'cev's Theorem, ${\rm SL}_n({\mathbb Z})$ is residually finite as well  \cite{Malcev}. It is known that for $n\geq 3$, the group ${\rm SL}_n({\mathbb Z})$ does not satisfy property FD
(see the main theorem in \cite{Bekka}, or \cite[Theorem 3.1]{Lubotzky} for a more general version of the result). This finishes the proof. 
\end{proof}

\begin{rmk}
   For $n< \infty$, we have $R^n_{B(\G)}:B(\G)\to B(G_n)$ is surjective, as $\G\setminus G_\infty$ is a group bundle with discrete unit space. The discussion presented in this section dealt with the restriction map  $R_{B(\G)}:B(\G)\to B(G_\infty)$.
\end{rmk}

\section{Applications: Box decomposition of the Fourier and Fourier-Stieltjes algebras}
\label{sec:boxsum}
Let $\G=\sqcup_{t\in T}{G_t}$ be a locally compact Hausdorff group bundle  (see \Cref{exa:groupoidexample} (iii)). Recall that $\unsp$ can be identified with $T$. Suppose that $\unsp$, when equipped with the subspace topology, is discrete. For each $t\in T$, let $u_t\in B(G_t)$, and define the function
\begin{equation}\label{eq2}
    \boxplus_{t\in T}u_t:\G \to \C, \ \left(\boxplus_{t\in T}u_t\right)(x)=u_t(x), \mbox{ if } x\in G_t.
\end{equation}
Next, we define 
$$\ell^\infty\text{-}\boxplus_{t\in T}B(G_t):=\set{\boxplus_{t\in T}u_t:\sup_{t\in T} \norm{u_t}_{B(G_t)}<\infty}$$
with componentwise scalar multiplication, addition, and multiplication. We equip $\ell^\infty\text{-}\boxplus_{t\in T}B(G_t)$ with the norm $\norm{\boxplus_{t\in T}u_t}_\boxplus:=\sup_{t\in T} \norm{u_t}_{B(G_t)}$.

\begin{prop}\label{prop:boxsumisometry1}
    Let $\G$ be a group bundle formed from the collection of locally compact Hausdorff groups $\set{G_t}_{t\in T}$. Suppose $\G$ is locally compact and Hausdorff, and $\unsp$ is discrete. Then, $B(\G)$ is isometrically (Banach algebra) isomorphic to $\ell^\infty\text{-}\boxplus_{t\in T}B(G_t)$.
\end{prop}
\begin{proof}
    For $t\in T$, we denote the restriction map $B(\G)\to B(\G_t^t)$ from \Cref{thm:RestrictionTheorem} by $R^t_{B(\G)}$. Note that $\G^t_t=G_t$, since $\G$ is a group bundle.
    Next, we define the map 
    $$\Phi: B(\G)\to \ell^\infty\text{-}\boxplus_{t\in T} B(G_t), \quad \varphi \mapsto \boxplus_{t\in T} R^t_{B(\G)}(\varphi).$$ 
    \begin{claim}\label{claim-well-defined}
        $\Phi$ is well-defined and algebra homomorphism. 
    \end{claim}
    \noindent \emph{Proof of \Cref{claim-well-defined}:}
    The claim that $\Phi$ is an algebra homomorphism follows from the fact that the restriction maps $R_{B(G_t)}^t$ are algebra homomorphisms. We only need to show that $\Phi$ is well-defined, in the sense that it has the correct codomain.
    Let $\omega\in B(\G)$. By \Cref{thm:RestrictionTheorem} \eqref{item-1-thm1}, 
    we have $R^t_{B(\G)}(\omega)\in B(G_t)$, for every $t\in T$. To show $\Phi$ is well-defined, it only remains to show that $\sup_{t\in T}\|R^t_{B(\G)}(\omega)\|_{B(G_t)}$ is finite. By definition of $B(\G)$, we can write $\omega=\calpi_{\bxi,\boldeta}$, where $\calpi$ is a representation of $\G$ on a continuous Hilbert field $\fH$, and $\bxi,\boldeta \in \Delta_b(\fH)$. 
    By \Cref{thm:RestrictionTheorem} \eqref{item-0-codomain}, for each $t\in T$ the restriction of $\calpi$ to $G_t$ is a continuous unitary representation of $G_t$. Let $\pi^t$ denote this representation, and set $\xi_t,\eta_t$ to be $\bxi(s(t))$ and $\boldeta(r(t))$, respectively. Then, for $x\in G_t$,
    $$R^t_{B(\G)}(\omega)(x)=\calpi_{\bxi,\boldeta}(x)=\innprod{\calpi(x)\bxi(t),\boldeta(t)}=\pi^t_{\xi_t,\eta_t}(x).$$
So 
\begin{eqnarray*}
\sup_{t\in T}\|R^t_{B(\G)}(\omega)\|_{B(G_t)}=
\sup_{t\in T}\|\pi^t_{\xi_t,\eta_t}\|_{B(G_t)}
\leq \sup_{t\in T}\norm{\xi_t}_{\Hil_{\pi^t}}\norm{\eta_t}_{\Hil_{\pi^t}}
\leq\norm{\bxi}_\Delta\norm{\boldeta}_\Delta<\infty.
\end{eqnarray*}
This proves that $\Phi$ is well-defined. 

\begin{claim}\label{claim-surjective}
    $\Phi$ is surjective. 
\end{claim}
\noindent \emph{Proof of \Cref{claim-surjective}:}
Let $\boxplus_{t\in T}u_t$ be an arbitrary element of $\ell^\infty\text{-}\boxplus_{t\in T}B(G_i)$. 
By \cite[Lemma 2.14]{Eymard}, each $u_t\in B(G_t)$ can be represented as a coefficient function $\pi^t_{\xi_t,\eta_t}$, where $\pi^t$ is a continuous unitary representation of $G_t$ on the Hilbert space $\Hil_t$, and $\xi_t,\eta_t\in \Hil_t$ such that $$\norm{u}_{B(G_t)}=\norm{\xi_t}\norm{\eta_t}.$$
Without loss of generality, we may assume $\norm{\eta_t}=1$ for each $t\in T$. 

Consider the continuous Hilbert field $\fH=\{\Hil_t\}_{t\in T}$, and define sections 
$\bxi, \boldeta$ as $\bxi(t)=\xi_t$ and $\boldeta(t)=\eta_t$. 
Clearly, the sections $\bxi, \boldeta$ are continuous as $T$ is discrete. 
As $\norm{\eta_t}=1$ for every $t\in T$, we have $\norm{\boldeta}_\Delta=1$. We also have
$$\norm{\bxi}_\Delta=\sup_{t\in T}\norm{\xi_t}=\sup_{t\in T}\norm{\xi_t}\norm{\eta_t}=\sup_{t\in T}\norm{u_t}<\infty,$$
since $\boxplus_{t\in T}u_t$ belongs to $\ell^\infty\text{-}\boxplus_{t\in T}B(G_i)$.
We now define the groupoid representation $\calpi$ of $\G$ on the continuous Hilbert field $\fH$ to be 
$$\calpi((t,g))=\pi^{t}(g).$$ 
The continuity of coefficient functions associated with $\calpi$ easily follows from the continuity of the coefficient functions of each $\pi^t$, since a converging net
$\set{(t_\alpha,g_\alpha)}_{\alpha\in I}$ in $\G$ eventually lies in one fiber $G_t$.
Finally, it is easy to see that $\Phi(\calpi_{\bxi,\boldeta})=\boxplus_{t\in T}u_t$, and $\Phi$ is surjective.

\begin{claim}\label{claim:isometry}
    $\Phi$ is an isometry.
\end{claim}
\noindent \emph{Proof of \Cref{claim:isometry}:}
Let $\varphi\in B(\G)$. By \Cref{thm:RestrictionTheorem}, each $R^t_{B(\G)}$ is norm-decreasing. So,
\begin{equation}\label{eq-norm-dec}
\norm{\Phi(\varphi)}_{\ell^\infty\text{-}\boxplus_{t\in T}B(G_i)}=\sup_{t\in T}\norm{R^t_{B(\G)}(\varphi)}_{B(G_t)}\leq \sup_{t\in T}\norm{\varphi}_{B(\G)}=\norm{\varphi}_{B(\G)}.
\end{equation}
On the other hand, we have
\begin{align*}
    \norm{\varphi}_{B(\G)}=\inf_{\varphi=\calpi_{\bxi,\boldeta}}\norm{\bxi}_\Delta \norm{\boldeta}_\Delta=\nonumber &\inf_{\varphi=\calpi_{\bxi,\boldeta}}\of{\sup_{t\in T}\norm{\bxi(t)}_{\Hil_t}\cdot \sup_{t\in T}\norm{\boldeta(t)}_{\Hil_t}}\\\nonumber
    &\geq\inf_{\varphi=\calpi_{\bxi,\boldeta}}\of{\sup_{t\in T}\norm{\bxi(t)}_{\Hil_t}\norm{\boldeta(t)}_{\Hil_t}}\\
    &\geq \sup_{t\in T}\of{\inf_{\varphi=\calpi_{\bxi,\boldeta}}\norm{\bxi(t)}_{\Hil_t}\norm{\boldeta(t)}_{\Hil_t}},
    \end{align*}
where the last inequality follows from \cite[Corollary 3.8.4]{CategoryTheory}.
So, 
\begin{eqnarray}\label{eq-norm-upper}
    \norm{\varphi}_{B(\G)}&\geq& \sup_{t\in T}\of{\inf_{\varphi=\calpi_{\bxi,\boldeta}}\norm{\bxi(t)}_{\Hil_t}\norm{\boldeta(t)}_{\Hil_t}}\nonumber\\
    &\geq& 
    \sup_{t\in T}\inf_{\varphi=\calpi_{\bxi,\boldeta}} \|R_{B(\G)}^t(\varphi)\|_{B(G_t)}\nonumber\\
    &\geq& 
    \sup_{t\in T} \|R_{B(\G)}^t(\varphi)\|_{B(G_t)}.
\end{eqnarray}
Putting \eqref{eq-norm-dec} and \eqref{eq-norm-upper} together, we get that $\Phi$ is an isometry.
\end{proof}
We apply \Cref{prop:boxsumisometry1} to show an isometric 
isomorphism $A(\G)\simeq c_0\text{-}\boxplus_{t\in T}A(G_t)$ where $c_0\text{-}\boxplus_{t\in T}A(G_t)$ is a subset of $\ell^\infty\text{-}\boxplus_{t\in T}B(G_t)$ consisting of elements $\boxplus_{t\in T} f_t$ with $f_t\in A(G_t)$, for which the function $t\mapsto \norm{f(t)}_{A(G_t)}$ vanishes at infinity.
In the following proof, we use $c_c\text{-}\boxplus_{t\in T}A(G_t)$, which is the subspace of $c_0\text{-}\boxplus_{t\in T}A(G_t)$ consisting of $\boxplus_{t\in T} f_t$ where only finitely many terms are nonzero.
\begin{cor}\label{cor:boxsumisometry2}
Let $\G=\sqcup_{t\in T} G_t$ be a  locally compact and Hausdorff  group bundle, and suppose that $\unsp$ is discrete. Then, $A(\G)$ is isometrically (Banach algebra) isomorphic to $c_0\text{-}\boxplus_{t\in T}A(G_t)$.
\end{cor}

\begin{proof}
    The map $\Phi:B(\G)\to \ell^\infty\text{-}\boxplus_{t\in T}B(G_t)$ of \Cref{prop:boxsumisometry1} provides an isometric isomorphism of Banach algebras. To prove the corollary, we only need to show that $\Phi(A(\G))=c_0\text{-}\boxplus_{t\in T}A(G_t)$.
    Recall that $A_{cf}(G_t)$ is the set of coefficient functions $L_{f,g}^{G_t}$ with $f,g\in C_c(G_t)$, and $A_c(G_t)$ is the algebra generated by $A_{cf}(G_t)$.  Let $u=\Lambda_{f,g}$, with $f,g\in C_c(\G)$, be an arbitrary element of $A_{cf}(\G)$. Since $T$ is discrete, there is a finite set $F\subseteq T$ such that $R^t_{A(\G)}(u)$ is nonzero for $t\in F$. Moreover, each $f|_{G_t}\in C_c(\G_t)$, and so $R^t_{A(\G)}(\Lambda_{f,g})\in A_{cf}(G_t)$. This demonstrates 
    $$\Phi(A_{cf(\G)})=c_c\text{-}\boxplus_{t\in T} A_{cf}(G_t).$$
    By \Cref{thm:denserange}, $R^t_{A(\G)}:A(\G)\to A(G_t)$ is a surjective algebra homomorphism.  So, we have
    $$\Phi(A_c(\G))=c_c\text{-}\boxplus_{t\in T}A_c(G_t).$$ 
    Then, by definition of $A(\G)$ and the fact that $\Phi$ is an isometry, we have
    $$\Phi(A(\G))=\Phi(\overline{A_c(\G)})=\overline{\Phi(A_c(\G))}=\overline{
    c_c\text{-}\boxplus_{t\in T}A_c(G_t)}=\overline{c_c\text{-}\boxplus_{t\in T}\overline{A_c(G_t)}}=\overline{c_c\text{-}\boxplus_{t\in T}A(G_t)}.$$
    Thus, $\Phi:A(\G) \to \overline{c_c\text{-}\boxplus_{t\in T}A(G_t)}=c_0\text{-}\boxplus_{t\in T}A(G_t)$ is an isometric algebra homomorphism.
\end{proof}

\section*{Acknowledgements}
The authors sincerely thank Brian Forrest for fruitful discussions and insightful ideas that led to a simplified proof for a result in this manuscript. 
The first author was partially supported by National Science Foundation Grant DMS-1902301 during this project.
The second author acknowledges support from National Science Foundation Grant DMS-1902301 and DMS-2408008 during the preparation of this article. 
The authors acknowledge Simons Foundation Travel Support for Mathematicians that funded a trip for their collaboration. 

\providecommand{\bysame}{\leavevmode\hbox to3em{\hrulefill}\thinspace}
\providecommand{\MR}{\relax\ifhmode\unskip\space\fi MR }
\providecommand{\MRhref}[2]{%
  \href{http://www.ams.org/mathscinet-getitem?mr=#1}{#2}
}
\providecommand{\href}[2]{#2}

\end{document}